\documentclass[12pt,twoside]{amsart}

\usepackage{hyperref}
\usepackage{amsthm, amsmath, amscd, amssymb,centernot}
\usepackage[all]{xy}
\usepackage{palatino,eulervm}
\usepackage[T1]{fontenc}
\usepackage[left=2cm,top=2.5cm,bottom=3cm,right=3cm]{geometry}
\newcommand{\str}{\mathcal{O}}

\newcommand{\proj}{\mathbb{P}}

\newcommand{\complex}{\mathbb C}

\theoremstyle{plain}
\numberwithin{equation}{section}
\newtheorem{theorem}{Theorem}[section]
\newtheorem*{theorem*}{Theorem}
\newtheorem{proposition}[theorem]{Proposition}

\newtheorem{corollary}[theorem]{Corollary}

\newtheorem*{conjecture*}{Nagata Conjecture}
\newtheorem*{conjecture1*}{SHGH Conjecture}
\theoremstyle{definition}

\newtheorem{remark}[theorem]{Remark}
\newtheorem{example}[theorem]{Example}

\begin{document}

\title{Positivity of line bundles on special blow ups of $\proj^2$}
\author[Krishna Hanumanthu]{Krishna Hanumanthu}
\thanks{Author was partially supported by a grant from Infosys Foundation}
\address{Chennai Mathematical Institute, H1 SIPCOT IT Park, Siruseri,
  Kelambakkam 603103, India}
\date{January 3, 2017}
\email{krishna@cmi.ac.in}
\maketitle

\begin{abstract}
Let $C \subset \proj^2$ be an irreducible and reduced curve of degree
$e$. 
Let $X$ be the blow up of $\proj^2$ at $r$ distinct smooth points
$p_1,\ldots,p_r \in C$. Motivated by results in \cite{H1,H2,VT}, we study
line bundles on $X$ 
and establish conditions for ampleness and $k$-very ampleness. 
\end{abstract}

\section{Introduction}
Let $X$ denote the blow up $\proj^2$ at $r$ points $p_1,\ldots,p_r \in
\proj^2$. It is interesting to ask when a given line bundle $L$ on $X$
has positivity properties such as 
ampleness, very ampleness,  global generation, and more generally, {\it
  $k$-very ampleness} (see the definition below). If the points
$p_1,\ldots,p_r$ are
general in $\proj^2$, this question has been extensively studied
and is related to important conjectures in algebraic
geometry. See \cite{Han,ST} for a detailed introduction and some results
in this case.

There has also been some work on these questions when the points are
special in some way. 
In \cite{H1,H2}, Harbourne considered blow ups of $\proj^2$ at $r$ 
points (not necessarily distinct) on an irreducible and reduced plane
cubic and a characterization of
line bundles with various positivity properties (ample, global
generated, effective, very ample) was given. De Volder \cite{DV}
partially generalized results of \cite{H1,H2} by
considering blow ups of points (not necessarily distinct) on a reduced and irreducible curve of
degree $e \ge 4$ and giving sufficient conditions for global generation
and very ampleness. 

Let $k$ be a non-negative integer. A line bundle $L$ on a projective
variety $X$ is said to be {\it $k$-very ample} if the restriction map 
$$H^0(X,L) \to H^0(X,L\otimes \str_Z)$$
is surjective for all zero-dimensional subschemes $Z \subset X$ of
length $k+1$; in other words, for all zero-dimensional subschemes $Z$
such that dim$(H^0(Z,\str_Z)) = k+1$. 

Note that 0-very ampleness is equivalent to global generation and
1-very ampleness is equivalent to very ampleness. 
As a result, $k$-very ampleness is considered a more general
positivity property for a line bundle. 
See \cite{BFS,BS1,BS2} for more details on the notion of $k$-very
ampleness.

A general theorem for $k$-very ampleness on
blow ups of projective varieties was proved by Beltrametti and Sommese in
\cite{BS3}.  De Volder and Tutaj-Gasi\'{n}ska \cite{VT} study $k$-very ampleness
for line bundles on blow ups of $\proj^2$ at general points on an
irreducible and reduced cubic. Szemberg and Tutaj-Gasi\'{n}ska \cite{ST} study $k$-very ampleness
for line bundles on blow ups of $\proj^2$ at general points. 
The property of $k$-very ampleness is also studied for other
classes of surfaces as well as higher-dimensional varieties. See \cite{BS,F}, for instance.

Our primary motivation comes from \cite{H1,H2,VT}. These papers study
positivity questions when $X$ is the blow up of $\proj^2$ at points on
a plane cubic. In this paper we generalize some of these results by considering 
blow ups of
$\proj^2$ at $r$ {\it distinct} and  {\it smooth} points on an irreducible and reduced plane
curve of degree $e$. 
More precisely, let $C \subset \proj^2$ be an irreducible and reduced curve of degree
$e$. Let $p_1,\ldots,p_r$ be distinct smooth points on $C$. We consider the
blow up $\pi:X \to \proj^2$ of $\proj^2$ at $p_1,\ldots,p_r$. Let $H$
denote the pull-back of $\str_{\proj^2}(1)$ and let $E_1,\ldots,E_r$
be the inverse images of $p_1,\ldots,p_r$ respectively. Given a line
bundle $L=dH-m_1E_1-m_2E_2-\ldots-m_rE_r$ on $X$, 
we are concerned with conditions on $d,e,r,m_1,\ldots,m_r$ which ensure
ampleness and $k$-very ampleness of $L$. 

In Section \ref{ample}, we study ampleness and prove our main result
(Theorem \ref{main}) in this case. 
In Section
\ref{k-very-ample}, we study
$k$-very ampleness. 
Here our main result is Theorem \ref{main1} which gives conditions for
$k$-very ampleness.

We work throughout over the complex number field $\complex$. 

\section{Ampleness}\label{ample}

The following is our main theorem on ampleness. 
\begin{theorem} \label{main} Let $C$ be an irreducible and reduced plane curve of
  degree $e$. Let $X \to \proj^2$ be the blow up of $\proj^2$ at $r$
  distinct smooth points $p_1,\ldots,p_r \in C$.  Let $H$
denote the pull-back of $\str_{\proj^2}(1)$ and let $E_1,\ldots,E_r$
be the inverse images of $p_1,\ldots,p_r$ respectively.
Let $L$ be a line bundle on $X$ with $L\cdot E_i > 0$ for all 
$1\le i \le r$.
Let $C_1$ denote the proper transform of $C$ on $X$.

If $L\cdot C_1 > 0$ and 
 $L \cdot H > L\cdot (E_{i_1}+\ldots+E_{i_e})$ for any $e$ distinct indices
 $i_1,\ldots,i_e \in \{1,\ldots,r\}$, then $L$ is ample. 
\end{theorem}
\begin{proof}

 Let $d: = L\cdot H >0$ and
$m_i := L \cdot
E_i > 0$ for every $1\le i \le r$.  By permuting the points, if necessary, assume that
$m_1 \ge \ldots \ge m_r$. So $L = dH-m_1E_1-\ldots-m_rE_r$ and $d >
m_1+\ldots+m_e$, by hypothesis. Also, we have $C_1 = eH-E_1-\ldots-E_r$.

We use the Nakai-Moishezon criterion for ampleness. 
We first verify that $L$
meets all irreducible curves $D$ on $X$ positively. 
If
$D=C_1$, then $L \cdot D >0$ by hypothesis.

So assume
that $D \ne C_1$. Write $D = fH-n_1E_1-\ldots-n_rE_r$ 
with $f \ge 0$. If $f=0$, then $D = -n_1E_1-\ldots-n_r E_r$ is
effective and this implies that $D=E_i$ for some $i$. Indeed, first
note that not all $n_i$ can be non-negative, because in that case $D$
is negative of an effective curve. If $n_i < 0$ for some $i$, then $D
\cdot E_i = n_i < 0$. So $E_i$ is a component of $D$ and hence
$D=E_i$. By hypothesis, $L \cdot E_i = m_i > 0$.  

Assume now that $f > 0$. Since $D \ne E_i$ for any $i$, it follows
that $n_i \ge 0$ for all $i$. 
Since $D \ne C_1$, we have $ef \ge
n_1+\ldots+n_r$.

By hypothesis, $d > m_1+\ldots+m_e$. Hence we have 
\begin{eqnarray*}
df &>& m_1f+\ldots+m_ef \\
&=& m_1n_1+\ldots+m_en_e+m_1(f-n_1)+\ldots+m_e(f-n_e)\\
&\ge& m_1n_1+\ldots+m_en_e+m_{e+1}(ef-n_1-n_2-\ldots-n_e)  \\ 
&\ge& m_1n_1+\ldots+m_en_e+m_{e+1}(n_{e+1}+\ldots+n_r)  \\
&\ge&
      m_1n_1+\ldots+m_en_e+m_{e+1}n_{e+1}+m_{e+2}n_{e+2}+\ldots+m_rn_r.
%      \label{2}. 
\end{eqnarray*}

Hence $L\cdot D > 0$.

The condition $L^2 > 0$ follows now because of 
Proposition \ref{standard-is-effective}, which says that $L$ is
effective if $L\cdot H \ge L\cdot (E_1+\ldots+E_e)$. 
\end{proof}

\begin{proposition}\label{standard-is-effective}
With $X$ as in Theorem \ref{main}, let $L = dH-m_1E_1-\ldots-m_rE_r$
be a line bundle on $X$ with $m_1\ge m_2 \ge \ldots \ge m_r$.
If $d \ge m_1+\ldots+m_e$, then $L$ is effective. 
\end{proposition}
\begin{proof}
This proposition is completely analogous to \cite[Lemma 1.4]{H1} and has
the same proof. We give a proof for completeness. 

We show that $L$ can be written as a non-negative linear combination
of line bundles of the form $H$, $H-E_1$, $2H-E_1-E_2,\ldots,
(e-1)H-E_1-\ldots-E_{e-1}$, $eH-E_1-\ldots-E_i$ for
$i=e,e+1,\ldots,r$. Since all these are effective, so is $L$. 

Let $s$ be the largest index in $\{1,2,\ldots,r\}$ such that $m_s \ne 0$. 
If $s \le e-1$, we 
have 

$L =
m_s(sH-E_1-\ldots-E_s)+(m_{s-1}-m_s)\left((s-1)H-E_1-\ldots-E_{s-1}\right)+
\ldots+(m_1-m_2)(H-E_1)+(d-m_1-m_2-\ldots-m_s)H$.

If $s \ge e$, let $L' =
L-m_s(eH-E_1-\ldots-E_s)$. Then it is easy to see that $L'$ satisfies
the hypotheses of the proposition and the value of $s$ is smaller for it. So we are done
by induction on $s$. 
\end{proof}
\begin{remark}\label{collinear}
In Theorem \ref{main}, the only hypothesis which is not in general
necessary for ampleness is the condition that $L\cdot H >
L\cdot(E_1+\ldots+E_e)$. But this condition is necessary if $e$ points
among $p_1,\ldots,p_r$ are collinear (assume that $e < r$). For instance, consider a line
$l$ in $\proj^2$
that meets $C$ in $e$ distinct smooth points, say $p_1,\ldots,p_e$.
Then choose any other $r-e$
points. If $L$ is ample, then $L$ meets the proper transform of $l$
positively. So we have $L\cdot H >
L\cdot(E_1+\ldots+E_e)$ after permuting the exceptional divisors, if
necessary. Thus in order to be ample for {\it all} choices of $r$
distinct smooth points on $C$, $L$ must satisfy $L\cdot H >
L\cdot(E_1+\ldots+E_e)$ after a suitable permutation of 
$E_1,\ldots,E_r$.

\end{remark}
In Corollary \ref{big-r}, we address the {\it
  uniform} case (i.e., $m_1=\ldots=m_r=m$), where we can make more
precise statements. 

\begin{corollary}\label{big-r}
Let $X$ be as in Theorem \ref{main}. Let $L
= dH-m\sum_{i=1}^r E_i$ be a
line bundle on $X$. If $r \ge e^2$, $L$ is ample if and only if $L \cdot C_1 >
0$. If $r < e^2$ then $L$ is ample if $d> em$. 
\end{corollary}
\begin{proof}
The proof is immediate from Theorem \ref{main}. 
Indeed, when $r \ge e^2$, the hypothesis gives $de>rm$ which implies
$d> \frac{r}{e}m \ge m$. When $r<e^2$, the hypothesis gives $de>e^2m>
rm$. 
\end{proof}

We note that when $r < e^2$ the condition $d>em$ is also {\em
  necessary} if $e$ of the points are collinear. See Remark
\ref{collinear}.

\begin{corollary}\label{nef}
With the set-up as in Theorem \ref{main}, $L$ is nef if $L\cdot C_1
\ge 0$ and $L\cdot H \ge L\cdot (E_1+\ldots+E_e)$. Further, if $r \ge
e^2$ and $L =  dH-m\sum_{i=1}^r E_i$, then $L$ is nef if and only if
$L\cdot C_1 \ge 0$. 
\end{corollary}
\begin{proof}
For
nefness, we only need to check that $L \cdot D \ge 0$ for all
effective curves. This is immediate from the proof of Theorem
\ref{main} and Corollary \ref{big-r}.
\end{proof}

\begin{remark}
In \cite{H1,H2}, Harbourne considers the case $e=3$. He defines a line
bundle $L=dH-m_1E_1-\ldots-m_rE_r$ to be {\it standard} with respect to
the {\it exceptional configuration} $\{H,E_1,\ldots,E_r\}$ if $d \ge
m_1+m_2+m_3$. Our hypothesis that $d\ge m_1+\ldots+m_e$ may be
considered as a generalization of the notion of standardness to the
case of arbitrary $e$. Harbourne defines $L$ to be {\it excellent}
if it is standard and $L\cdot C_1 > 0$. Suppose that $m_i > 0$ for
every $i$. One of the main results in
\cite{H1,H2} says that $L$ is ample {\em if and only if} it is excellent
with respect to some exceptional configuration. See \cite{H1,H2} for
more details. 

Our main Theorem \ref{main} may be considered as a generalization of
one direction of this result to the case of arbitrary $e$. The converse is not true
when $e \ge 4$. See Example \ref{e=4}.
\end{remark}

\begin{example}\label{e=3} 
Take $e=3$ and $r=10$. Let $L =
8H-3(E_1+E_2+E_3)-2(E_4+\ldots+E_{10})$. Consider  the following transformation:

$H\mapsto 2H-E_1-E_2-E_3$, 
$E_1\mapsto H-E_2-E_3$,
$E_2\mapsto H-E_1-E_3$, 
$E_3\mapsto H-E_1-E_2$, and 
$E_i \mapsto E_i$ for $i = 4,\ldots,10$.

Under this, $L$ is transformed to $7H-2(E_1+\ldots+E_{10})$. This is standard, in
fact excellent, in
the sense of \cite{H1,H2}. So $L$ is ample. One can also check that
$L$ is ample as in Example \ref{e=4}.  This example illustrates one of
the main theorems in \cite{H1} which says that a line bundle $L$ is ample if and
only if $L\cdot C_1 > 0$ and $L$ is standard with respect to {\it some}
exceptional configuration.

\end{example}

\begin{example}\label{optimal}
Let $e=4$ and $r=17$. So $X$ is the blow up of $\proj^2$ at 17
distinct smooth points on an irreducible,
reduced plane quartic. Let $L =
11H-3(E_1+E_2+\ldots +E_{13})-(E_{14}+\ldots+E_{17})$. Then $L\cdot
C_1 = 44-39-4=1$, but $L^2=121-117-4=0$. So $L$ is not ample. 
Here note that $d=11 < m_1+m_2+m_3+m_4=12$. 
So the
hypotheses in Theorem \ref{main} can not be weakened. 
\end{example}

\begin{example}\label{e=4}
In this example, we show that the hypotheses of Theorem \ref{main} are
not always necessary for ampleness. 
Let $C$ be an irreducible and reduced plane quartic and let
$p_1,\ldots,p_{18}$ be distinct smooth points on $C$ such that no four are
collinear. 
Let $L =
10H-3(E_1+E_2+E_3)-2(E_4+\ldots+E_{18})$. Since $d= 10< 3+3+3+2=11$, the
hypotheses of Theorem \ref{main} are not
satisfied. However, we claim that $L$ is ample. 

It is easy to check that $L^2 = 13$ and $L\cdot C_1 = 1$. So let $D
\ne C_1$ be an irreducible and reduced curve. Write $D =
fH-\sum_{i=1}^{18} n_iE_i$. 

If $f=0$, we may assume that 
$n_i \ne 0$ for some $i$. We claim in fact that $n_i \le 0$ for all $i$. 
Since $D=-(\sum_i n_iE_i)$ is effective,
$n_i$ can not all be non-negative, as in that case $D$ is negative
of an effective divisor. So $n_i < 0$ for some $i$. Then $D\cdot E_i =
n_i < 0$, so that $E_i$ is a component of $D$. Subtracting $E_i$ from
$D$ for all $i$ with $n_i < 0$, we obtain an effective divisor of the
form $\sum\limits_{n_j > 0}(-n_j) E_j$, which must be the zero divisor.  
Hence $L \cdot D = \sum_i (-n_i) > 0$. 

Now let $f \ge 1$ and $n_1, 
\ldots, n_{18} \ge 0$. In fact, if $f=1$, since no four points are
collinear, we have  $L
\cdot D > 0$. So let $f \ge 2$.

Since $C_1$ and $D$ are distinct irreducible curves, $C_1 \cdot D = 4f
-\sum_{i=1}^{18} n_i \ge 0$. 
Then 
%\begin{eqnarray*} 
$L\cdot D =
10f-3(n_1+n_2+n_3)-2(n_4+\ldots+n_{18}) =  \left(8f-2\sum_{i=1}^{18}n_i\right)
+2f-(n_1+n_2+n_3) \ge 2f-(n_1+n_2+n_3).$
%\end{eqnarray*} 
If $n_1+n_2+n_3 = 0$, it
follows that $L \cdot D > 0$. Otherwise, without loss of generality,
let $n_1 > 0$. 
Intersecting $D$ with
the proper transforms of the line through $p_1,p_2$ and the line 
through $p_1,p_3$, we get $f \ge n_1+n_2$ and $f\ge n_1+n_3$. Hence 
$2f\ge 2n_1+n_2+n_3 > n_1+n_2+n_3$. The last inequality holds because $n_1
> 0$. Thus $L \cdot D > 0$. 

Though $L = 10H-3(E_1+E_2+E_3)-2(E_4+\ldots+E_{18})$ in
this example is ample, it is easy to see that $L$ does not satisfy the condition $d \ge
m_1+m_2+m_3+m_4$ with respect to {\em any} exceptional configuration. This
is easy to check by direct calculation. Note also that $L$ is already
standard in the sense of Harbourne \cite{H1,H2}. 
\end{example}

\section{$k$-very ampleness}\label{k-very-ample}

\iffalse
Let $k$ be a non-negative integer. A line bundle $L$ on a projective
variety is said to be {\it $k$-very ample} if the restriction map 
$$H^0(X,L) \to H^0(X,L\otimes \str_Z)$$
is surjective for all zero-dimensional subschemes $Z \subset X$ of
length $k+1$, i.e., dim$(H^0(Z,\str_Z)) = k+1$. 

Note that 0-very ampleness is equivalent to global generation and
1-very ampleness is equivalent to very ampleness. 

This notion has been introduced in a series of papers by Beltrametti,
Francia and Sommese. See \cite{BFS,BS1,BS2} for more details. 
  general theorem is proved in \cite{BS3} for $k$-very ampleness on
blown up projective varieties. \cite{VT} studies $k$-very ampleness
for line bundles on blow ups of $\proj^2$ at general points on an
irreducible and reduced cubic. \cite{ST} studies $k$-very ampleness
for line bundles on blow ups of $\proj^2$ at general points. 
\fi

Let $C, p_1,\ldots,p_r, X$ be as in Section \ref{ample}. In this
section we consider a 
line bundle $L =dH-m\sum_{i=1}^r E_i$ on
$X$ and 
investigate $k$-very ampleness of $L$ for a non-negative integer
$k$. We make the assumption that the number of points we blow up is
large compared to $e$. Specifically, we assume that $r \ge e^2+k+1$.

First, we consider the question of global generation.
In other words, we assume $k=0$. 
Our arguments for $k=0$ give a flavour of our arguments in the 
case $k \ge 1$, which we consider later. 

We will use Reider's theorem \cite[Theorem 1]{Re} which gives conditions for global
generation and very ampleness. We only state the conditions for global
generation below. 

\begin{theorem}[Reider]\label{reider}
Let $X$ be a smooth complex surface and let $N$ be a nef line bundle on $X$ with $N^2
\ge 5$. If $K_X+N$ is not
globally generated (here $K_X$ is the canonical line bundle of $X$), then there exists an effective divisor $D$ on $X$ such
that 
$$D\cdot N = 0, D^2=-1, {\rm ~or~} 
D\cdot N = 1, D^2 = 0.$$
\end{theorem}

The following is our theorem on global generation. 

\begin{theorem}\label{bpf}
Let $C$ be an irreducible and reduced plane curve of
  degree $e$ and let $X \to \proj^2$ be the blow up of $\proj^2$ at $r$
  distinct smooth points $p_1,\ldots,p_r \in C$.  Let $H$
denote the pull-back of $\str_{\proj^2}(1)$ and $E_1,\ldots,E_r$
the inverse images of $p_1,\ldots,p_r$ respectively. 
Let $L =dH-m\sum_{i=1}^r E_i$ be a line bundle on $X$ with $m\ge 0$. 

If $(d+3)e>r(m+1)$ and $r \ge e^2+1$, then 
$L$ is globally generated.
\end{theorem}
\begin{proof}
%We use Reider's theorem. 
Since the conclusion holds if $m=0$, we
assume $m \ge 1$. 

Let $N = L-K = (d+3)H-(m+1)\sum_{i=1}^r E_i$. 
%By Theorem \ref{main}, $N $ is ample. Indeed, 
If $C_1$ denotes the proper transform of $C$ on
$X$, then $N\cdot C_1 = (d+3)e-r(m+1) > 0$, by hypothesis. Moreover,
$d+3 > \frac{r}{e}(m+1) > e(m+1)$, since $r> e^2$. So the
hypotheses of Theorem \ref{main} hold and $N$ is
ample. 

Further, $N^2 = (d+3)^2-r(m+1)^2 > \left(\frac{r^2}{e^2}-r\right)(m+1)^2 >
(m+1)^2 \ge 4$. This is because $\frac{r^2}{e^2}-r\ge
r(\frac{e^2+1}{e^2})-r = r(1+\frac{1}{e^2})-r = \frac{r}{e^2} >
1$. So we can apply Reider's
Theorem \ref{reider}. 

Suppose that $L$ is not globally generated. 
By Theorem \ref{reider}, there is an effective divisor $D$
such that $D\cdot N = 1, D^2 = 0$.  
Since $N$ is ample, this is the only 
possibility. Writing $D = fH-\sum_{i=1}^r n_iE_i$ and setting $n = \sum_{i=1}^r n_i$, 
we have $$D \cdot N = (d+3)f - n(m+1) = 1 \text{~and~} D^2 = f^2 -\sum_{i=1}^r
n_i^2 = 0.$$

Note that if $n_i < 0$ for some $i$, then $D \cdot E_i < 0$. Thus $E_i$ is a
component of $D$ and $D-E_i$ is effective. But then we get a
contradiction because $1 = D\cdot N =
N\cdot (D-E_i) + N\cdot E_i \ge m+1 \ge 2$ (since $N$ is ample). Hence
$n_i \ge 0$ for all $i$ and $n =\sum_i n_i \ge 0$. Since $D^2 =0$, in fact $n > 0$. 

We consider two different cases: $n \ge r$ or $n < r$. 

First, suppose that $n \ge r$. Since $f^2
= \sum_i n_i^2 \ge \frac{n^2}{r}$ and by hypothesis, $d+3 >
\frac{r}{e}(m+1)$,  we have 
\begin{eqnarray}\label{3.1-a}
1 = D\cdot N = (d+3)f-n(m+1) >
\frac{n\sqrt{r}}{e}(m+1)-n(m+1).  
\end{eqnarray}

We claim that  $\frac{n\sqrt{r}}{e}-n \ge 1/2$ for any fixed $e$, all
$r\ge e^2+1$
and for all $n \ge r$. Since this is a linear function in $n$ of
positive slope, it suffices to show that the function is non-negative
when $n=r$. That is, we only have to show that $\frac{(\sqrt{r})r}{e}-r
\ge 1/2$ for $r \ge e^2+1$. For a fixed $e$, this function is
increasing for $r > 0$. So it suffices to show that 
$\frac{\sqrt{e^2+1}(e^2+1)}{e}-(e^2+1) \ge 1/2$. It is easy to 
see that this inequality holds for $e \ge 1$, for example by
clearing the denominator and squaring. 
%A more general version of this case is done in Proposition \ref{n-greater-than-r}. 

Thus, by \eqref{3.1-a},  $ 1>\left(\frac{n\sqrt{r}}{e}-n\right)(m+1)
\ge \frac{m+1}{2} \ge 1$, which is a contradiction. 

Finally, we consider the case $n < r$. We have $f^2 = \sum_{i=1}^r n_i^2 \ge
n$. So $1 = N\cdot D = (d+3)f-n(m+1) \ge \left(\frac{r\sqrt{n}}{e}
  -n\right)(m+1)$. We claim that $\frac{r\sqrt{n}}{e}
  -n \ge 1$, which as above leads to a contradiction. 

We view $\frac{r\sqrt{n}}{e} -n$ as a quadratic in $\sqrt{n}$. Since
the leading coefficient is -1, it is a downward sloping parabola. If
we show that the value of this function is at least 1 for $n=1$ and
$n=r-1$, then it follows that  the value of the function is at least 1
for all $1 \le n \le r-1$. This can be easily verified. 

We conclude that $L$ is globally generated. 
\end{proof}

Now we will consider the case $k \ge 1$. 
Recall the criterion \cite[Theorem 2.1]{BFS} of Beltrametti-Francia-Sommese for $k$-very
ampleness of $L$, which generalizes Reider's criterion in Theorem \ref{reider}. 

\begin{theorem}[Beltrametti-Francia-Sommese]\label{bfs-theorem}
Let $N$ be a line bundle on a surface $X$. Let $k \ge 0$ be an integer. 
Suppose that $N$ is nef and $N^2 \ge 4k+5$.
If $K_X+N$ is not
$k$-very ample, then there exists an
effective divisor $D$ on $X$ such that 
\begin{eqnarray}\label{bfs}
N\cdot  D - k-1 \le D^2 < \frac{N\cdot D}{2} < k+1.
\end{eqnarray}
\end{theorem}

If \eqref{bfs} holds for an effective divisor $D$, then our next two 
results give some conditions that $D$ must satisfy.

\begin{proposition} \label{n-greater-than-r}
Let $X$ be as in Theorem \ref{bpf}. Let $L =dH-m\sum_{i=1}^r E_i$ 
and $N = L-K_X =(d+3)H-(m+1)\sum_{i=1}^r E_i$ be line bundles on $X$.
Let $k\ge 1$ be an integer.
Suppose that $r \ge e^2+k+1$, $(d+3)e > r(m+1)$ and $m\ge k$. If an
effective
divisor $D=fH-\sum_{i=1}^r
n_iE_i$ on $X$ satisfies \eqref{bfs}, then $\sum_{i=1}^r n_i < r$. 
\end{proposition}
\begin{proof}
Set $n=\sum_{i=1}^r n_i$. 
Let $\alpha = N\cdot D = (d+3)f-(m+1)n$ and $\beta = D^2 = f^2 -
\sum_{i=1}^r n_i^2$. If $D$ satisfies \eqref{bfs}, then we have 
\begin{eqnarray}\label{bfs-1}
\alpha - k-1 \le \beta < \frac{\alpha}{2} < k+1.
\end{eqnarray}

 By hypothesis,
$(d+3)^2 > \frac{r^2}{e^2}(m+1)^2$. Since $(d+3)f=\alpha+(m+1)n$, we have 

$$(d+3)^2f^2  = (m+1)^2n^2+2(m+1)n\alpha+\alpha^2 >
\frac{r^2}{e^2}(m+1)^2\left(\beta+\sum_{i=1}^r n_i^2\right).$$
Since $r\sum_{i=1}^r n_i^2 \ge
n^2$, we have 
\begin{eqnarray*}
(m+1)^2n^2+2(m+1)n\alpha+\alpha^2 > \frac{r^2}{e^2}(m+1)^2\beta+\frac{r}{e^2}(m+1)^2n^2.
\end{eqnarray*}

We now show that the above inequality is impossible for $n \ge r$. 
Specifically, we make the following claim.

{\bf Claim:} Let $r,e,k,\alpha, \beta$ be as in the proposition and
suppose that \eqref{bfs-1} holds. Then for $n\ge r$, we have
$$\frac{r^2}{e^2}(m+1)^2\beta+\frac{r}{e^2}(m+1)^2n^2 \ge
(m+1)^2n^2+2(m+1)n\alpha+\alpha^2.$$

{\bf Proof of Claim:} We consider the difference of the two terms in
the required inequality as a quadratic function in $n$. 
Define 
$$\lambda(n) := 
\left(\frac{r}{e^2}-1\right)(m+1)^2n^2-2(m+1)n\alpha+\frac{r^2\beta}{e^2}(m+1)^2-\alpha^2.$$
This is quadratic in $n$ with
the leading coefficient $\left(\frac{r}{e^2}-1\right)(m+1)^2 > 0$. 
We will show that $\lambda(r) \ge 0$ and $\lambda'(r) \ge 0$, which will
prove the claim and the proposition. 

\begin{eqnarray*} 
\lambda(r) &=& r^2(m+1)^2\left(\frac{r+\beta}{e^2}-1\right)
-2(m+1)\alpha r-\alpha^2\\
&\ge& r^2(m+1)^2\left(\frac{e^2+k+1+\beta}{e^2}-1\right)
-2(m+1)\alpha r-\alpha^2   \hspace{.3in} ({\rm since~} r \ge e^2+k+1)\\
&=& r^2(m+1)^2\left(\frac{k+1+\beta}{e^2}\right)
-2(m+1)\alpha r-\alpha^2\\
&\ge& r(e^2+k+1)(m+1)^2\left(\frac{k+1+\beta}{e^2}\right)
-2(m+1)\alpha r-\alpha^2\\
&=& r(m+1)^2(k+\beta+1)+ \frac{r(k+1)(m+1)^2(k+\beta+1)}{e^2}-2(m+1)\alpha r-\alpha^2\\
&\ge&  r(m+1)^2(k+\beta+1)+ (k+1)(m+1)^2(k+\beta+1)-2(m+1)\alpha r-\alpha^2\\
&\ge& 0.
\end{eqnarray*}

The last inequality follows when we compare the first term with the
third term and the second term with the fourth term. We use the
inequalities 
$\alpha \le \beta+k+1$ and $\alpha < 2(k+1)$ which hold by \eqref{bfs-1}, 
and $m\ge k \ge 1$, which holds by hypothesis. 

Next we show that $\lambda'(r)\ge 0$. 

$\lambda'(n) = 2\left(\frac{r}{e^2}-1\right)
(m+1)^2n-2(m+1)\alpha$. Thus 
\begin{eqnarray*} 
\lambda'(r) &=& \left(\frac{2r^2}{e^2}-2r\right)
(m+1)^2-2(m+1)\alpha\\
&\ge& \frac{2r(e^2+k+1)}{e^2}(m+1)^2-2r(m+1)^2-2(m+1)\alpha\\
&=& 2r(m+1)^2+2r(m+1)^2\frac{k+1}{e^2}-2r(m+1)^2-2(m+1)\alpha\\
&=&2r(m+1)^2\frac{k+1}{e^2}-2(m+1)\alpha\\
&\ge& 2(m+1)^2(k+1)-2(m+1)\alpha\\
&\ge& 0 \hspace{.2in}(\text{by \eqref{bfs-1} and the hypothesis that $m \ge k \ge 1$}).
\end{eqnarray*}

This completes the proof of the proposition.
\end{proof}

\begin{proposition}\label{proper-with-c}
Let $L =dH-m\sum_{i=1}^r E_i$ and $N = L-K_X =(d+3)H-(m+1)\sum_{i=1}^r E_i$.
Let $k$ be a positive integer.
Suppose that $r \ge e^2+k+1$, $(d+3)e > r(m+1)$ and $m\ge k$. Let $D$
be an effective divisor on $X$ such that $D \cdot C_1 \ge 0$. Then $D$
does not satisfy \eqref{bfs}.
\end{proposition}

\begin{proof}
Let $D = fH-\sum_{i=1}^r n_iE_i$. Then $f \ge
0$ and if $f=0$, then $n_i \le 0$ for all $i = 1,\ldots, r$. 

Let $n = \sum_{i=1}^r n_i$. Then $D^2 = f^2 -\sum_{i=1}^r n_i^2 \le
f^2-n$. Indeed, this follows because $\sum_{i=1}^r n_i^2 \ge n$. 
Moreover, the assumption
$D\cdot C_1 \ge 0$ implies that $ef \ge n$. 

Suppose that $D$ satisfies
\eqref{bfs}. We will obtain a contradiction. 

%We observe that $D^2 = f^2 -\sum_{i=1}^r n_i^2 \le f^2-n$. Indeed, 
%this follows because $\sum_{i=1}^r n_i^2 \ge n$. 
%Further, since $D$ is effective, $f \ge 0$. 

%If $f=0$, we may assume that not
%all $n_i$ are zero. Moreover, we must have $n_i
%\le 0$ for all $i$. Indeed, 
%since $D=-\sum_i n_iE_i$ is effective,
%$n_i$ can not all be positive, since in that case $D$ is negative
%of an effective divisor. So $n_i < 0$ for some $i$. Then $D\cdot E_i =
%n_i < 0$, so that $E_i$ is a component of $D$. Subtracting $E_i$ from
%$D$ for all $i$ with $n_i < 0$, we obtain an effective divisor of the
%form $\sum\limits_{n_j > 0}-n_j E_j$, which is absurd.  

First let $f=0$. We have $N \cdot D = -(m+1)n < 2(k+1) \le 2(m+1)$. So 
$n > -2$. On the other hand,  $0 = ef \ge n$. So $n=0$ or $n=-1$. Since
each $n_i$ is non-positive, if $n=0$ then $D=0$. If $n=-1$, then
$D=E_i$ for some $i$. But then $N\cdot D = m+1 \ge k+1$, hence 
$N\cdot D -k-1 \ge 0$, while $D^2 = -1$. This violates \eqref{bfs}.
  
Let $f=1$. In this case, we have $e \ge n$. We may also
assume $n > 0$. 
Then $2(k+2) > N \cdot D = (d+3)-n(m+1) >
\left(\frac{r}{e}-n\right)(m+1) \ge (e-n)(m+1)+\left(\frac{k+1}{e}\right)(m+1) \ge
(e-n)(k+1)+\left(\frac{k+1}{e}\right)(k+1)$. Thus $0\le e-n< 2$. So $e=n$ or $e=n+1$. 

Let $e=n+1$. Then $0 \ge 1-n \ge D^2 \ge N\cdot D -k-1 >
k+1+\left(\frac{k+1}{e}\right)(k+1)-k-1> 0$, which is a contradiction. If $e=n$,
then  $0 \ge 1-n \ge D^2 \ge N\cdot D -k-1 >
\left(\frac{k+1}{e}\right)(k+1)-k-1$. Hence $k+1< e$. But then $D^2 \le 1-n = 1-e
< -k \le N\cdot D -k-1$. The last inequality holds because $N$ is
ample and hence $N \cdot D > 0$. Again we have a
contradiction, because this violates \eqref{bfs}. 

Now suppose that $f\ge 2$. As above, we have 
$2(k+1) > N\cdot D > \frac{rf(m+1)}{e}-n(m+1) \ge \frac{rf(m+1)}{e}-ef(m+1) =
\left(\frac{rf}{e}-ef\right)(m+1) \ge
\left(\frac{(k+1)f}{e}\right)(m+1).$
%\ge  \left(\frac{(k+1)f}{e}\right)(k+1)$. 

\begin{eqnarray}\label{3.4-a}
{\rm Thus~} 2 > \frac{(m+1)f}{e}, {\rm ~or~ equivalently,~} f < \frac{2e}{m+1}.
\end{eqnarray}

If $e \le k+1$, then $f <  \frac{2e}{m+1} \le  \frac{2e}{k+1} \le 2$,
which contradicts the hypothesis $f \ge 2$. So we may assume $e >
k+1$. 

We now make the following claim:

{\bf Claim:} $f^2 - n \le  2-2k$.  

{\bf Proof of Claim:} 
Note that $ef-n \le 1$. Indeed,
$(d+3)f>\frac{rf}{e}(m+1)>ef(m+1)$. Hence $2(k+1) > N\cdot D  =
(d+3)f-n(m+1) > (ef-n)(m+1)\ge (ef-n)(k+1)$. Thus $ef-n < 2$. On the
other hand, by the hypothesis in the proposition $ef-n \ge 0$. Hence
we have:

\begin{eqnarray}\label{3.4-b}
ef-n=0 {\text ~or~} ef-n=1.
\end{eqnarray}

On the other hand, 
\begin{eqnarray}
f^2 -n &<& f\left(\frac{2e}{m+1}\right)-n \hspace{.2in} ({\rm by~}
           \eqref{3.4-a}) \nonumber \\
&=& ef+\frac{ef(1-m)}{m+1}-n \nonumber\\
&\le& \frac{ef(1-k)}{k+1}+ef-n \hspace{.2in}({\rm ~since~} m \ge k)
      \nonumber \\
&\le& \frac{ef(1-k)}{k+1}+1 \hspace{.2in}({\rm ~by~} \eqref{3.4-b})
      \nonumber \\
%\end{eqnarray}
%\begin{eqnarray}
&\le& f(1-k)+1 \hspace{.2in} ({\rm ~since ~}e > k+1 {\rm
      ~and ~} 1-k \le 0) \nonumber \\
&\le& 2(1-k) +1 = 3-2k \hspace{.2in} ({\rm ~since ~} f\ge 2 {\rm
      ~and ~} 1-k \le 0). \nonumber
\end{eqnarray}

This completes the proof of the claim. Now we consider three cases:

$\underline{k \ge 3}$: In this case,  $f^2-n \le 2-2k \le -k-1$. So
$D^2 = f^2 - \sum_{i=1}^rn_i^2 \le f^2-n \le -k-1$. But 
by \eqref{bfs}, we have $N\cdot D -k-1 \le D^2$. Since $N$ is ample,
$-k \le N\cdot D -k-1$, contradicting the inequality $D^2 \le -k-1$.

$\underline{k=1}$:  Then by \eqref{bfs}, the claim, and the fact that $N \cdot D \ge
1$, we have
$-1 \le N\cdot D - 2 \le f^2-n \le 0$. 
%which implies $N\cdot D \le 2$. 
By \eqref{3.4-b}, we have either $ef=n$ or
$ef-1=n$. 

If $ef=n$, then $0 \ge f^2-n = f(f-e) \ge -1$. Since
$f \ge 2$, the only possibility $f=e$.  But this violates
\eqref{3.4-a}, because $m \ge k = 1$.
On other hand, if $ef-1=n$,
then $0 \ge f^2-n = f^2-ef+1=f(f-e)+1\ge -1$. Since $f \ge 2$, $f^2-n$
must be -1. 
But then the only possibility is $f=2$ and $e=1$ and again we have a contradiction to \eqref{3.4-a}.

$\underline{k=2}$: By \eqref{bfs}, the claim, and the fact that $N \cdot D \ge
1$, we have
$-2 \le N\cdot D - 3 \le f^2-n \le -2$. 
Hence $f^2-n=-2$. 
If $ef=n$, then $-2=f^2-ef=f(f-e)$. This contradicts \eqref{3.4-a}. 
If $ef-1 = n$, then $-2=f(f-e)+1$.  
Again we obtain a contradiction to
\eqref{3.4-a}.

This completes the proof of the proposition.
\end{proof}

Now we are ready to prove our main result on $k$-very ampleness. 

\begin{theorem}\label{main1}
Let $C$ be an irreducible and reduced plane curve of
  degree $e$. Let $X \to \proj^2$ be the blow up of $\proj^2$ at $r$
  distinct smooth points $p_1,\ldots,p_r \in C$.  Let $H$
denote the pull-back of $\str_{\proj^2}(1)$ and let $E_1,\ldots,E_r$
be the inverse images of $p_1,\ldots,p_r$ respectively. Let $k$ be a
non-negative integer. 
Let $L=dH-m\sum_{i=1}^r E_i$ be a line bundle on $X$ with $m
\ge k$.

 If $(d+3)e>r(m+1)$ and $r \ge e^2+k+1$, then $L$ is $k$-very ample.
\end{theorem}

\begin{proof}
When $k=0$, this is the same as Theorem \ref{bpf}. So we will assume that
$k \ge 1$ and use the criterion of Beltrametti-Francia-Sommese. 

Let $N = L-K = (d+3)H-(m+1)\sum_{i=1}^r E_i$. Just as in the proof of
Theorem \ref{bpf}, we conclude that $N$ is ample.

Next we claim that $N^2 \ge 4k+5$. Indeed, we have
\begin{eqnarray*}
N^2 &=& (d+3)^2-r(m+1)^2 \\
&>& \frac{r^2}{e^2}(m+1)^2-r(m+1)^2\\
&=&(m+1)^2r\left( \frac{r}{e^2}-1\right)\\
&\ge& (m+1)^2r\left(\frac{k+1}{e^2}\right)\\
%&\ge& (m+1)^2(e^2+k+1)\left(\frac{e^2+k+1}{e^2}-1\right)\\
%&\ge& (m+1)^2 \left( k+1+\frac{(k+1)^2}{e^2}\right)\\
%\end{eqnarray*}
%\begin{eqnarray*}
&>& (m+1)^2(k+1) \\
&\ge& 4(k+1) \hspace{.5in} ({\rm since~} m\ge k \ge 1).
\end{eqnarray*}

Since $N^2$ is an integer and $N^2 > 4k+4$, we conclude that $N^2 \ge
4k+5$. Hence we can apply the criterion of Beltrametti-Francia-Sommese.
Suppose that $L$ is not $k$-very ample. Then there exists an effective
divisor $D$ on $X$ such that \eqref{bfs} holds. We will obtain a
contradiction. 

Write $D = fH - n_1E_1-\ldots-n_rE_r$ 
and let  $n= \sum_i n_i$. By Proposition \ref{n-greater-than-r}, we have
$n<r$. By Proposition \ref{proper-with-c},
$D \cdot C_1 < 0$. This implies that $C_1$ is a component of $D$ and we
may write $D = aC_1 + D'$ for a positive integer $a$ and an effective
divisor $D'$. 

Write $D' = bH - \sum_{i=1}^r l_i E_i$ with $b \ge 0$. Then we have
$f=ae+b$. So 
\begin{eqnarray*}
2(k+1) &>& N\cdot D \\
&=& (d+3)f-(m+1)n \\
&>& \frac{r}{e}(ae+b)(m+1)-(r-1)(m+1)  \hspace{.2in} ({\rm since~}
       n\le r-1)\\
&=& (r(a-1)+1)(m+1)+\frac{rb}{e}(m+1)\\
&\ge&  (r(a-1)+1)(k+1)+\frac{rb}{e}(k+1).
\end{eqnarray*}

Thus $a=1$ and $b=0$. In particular,  $f=e$. %$D' = -\sum_{i=1}^r l_i E_i$.
We have $2(k+1) > N\cdot D = (d+3)e-n(m+1) > (r-n)(m+1)$. Hence $r-n <
2$. On the other hand, $r-n \ge 1$. So $r-n=1$. 

Since $N$ is ample, $N\cdot D > 0$. Thus $-k \le N\cdot D -k - 1\le
D^2 = e^2 -\sum_{i=1}^r n_i^2 \le e^2 -n = e^2-r+1 \le -k$. The last
  inequality holds because $r \ge e^2+k+1$. Thus we have $D^2 = -k $
  and $N\cdot D =1$. But $N\cdot D > m+1 \ge k+1 > 1$. This is a
  contradiction.  

The proof of the theorem is complete.
\end{proof}

\begin{remark}
\cite[Theorem 4.1]{VT} gives conditions for $k$-very ampleness for any
line bundle on the blow up of $\proj^2$ at general points on an
irreducible and reduced cubic. In our context, this is the case
$e=3$. If the line bundle is uniform, that is if $L = dH-m\sum_{i=1}^r
E_i$ and if the number of points $r$ is at least $10+k$, then our Theorem
\ref{main1} is comparable to this result. However, we note that \cite[Theorem
4.1]{VT} deals with any (not just uniform) line bundle $L = dH-\sum_{i=1}^r m_iE_i$
and any $r
\ge 3$. 
\end{remark}

The next two examples show that 
the hypothesis of
Theorem \ref{main1} can not be weakened for $e=3$.

\begin{example}\label{strict-inequality}
Let $C$ be a smooth plane cubic. Let $X$ be the blow up of 10 
distinct points on $C$. Consider the line bundle $L  =
7H-2\sum_{i=1}^{10} E_i$ on $X$. We
have $C_1 = 3H-\sum_{i=1}^{10} E_i$. By Corollary \ref{big-r}, $L$ is
ample. We use \cite[Corollary 1.4]{BS1} to show that $L$ is not 
globally generated. According to this result, if a line bundle on
a curve of positive genus is $k$-very ample, then the degree of the 
line bundle is 
at least $k+2$. In our example, if $L$ is globally generated (that is,
if it is 0-very ample), then $L_{|C}$ is also $0$-very ample on $C$. 
But deg$(L_{|C}) = L\cdot C =1 < 2$. 
So the strict inequality, $e(d+3) > r(m+1)$,  in the hypothesis of
Theorem \ref{main1} can not be relaxed. 
\end{example}

\begin{example} \label{points}
This is a small variation on Example \ref{strict-inequality}.
Again let $e=3, r= 10$, but now let  $m =7$. Consider $L  =
24H-7\sum_{i=1}^{10} E_i$. It is easy to check that $L$ is ample (by
Corollary \ref{big-r}) and globally generated (by Theorem
\ref{main1}). But $L$ is not very ample because $L\cdot C_1 = 2 <
1+2$ (again by \cite[Corollary 1.4]{BS1}). Here the hypothesis in
Theorem \ref{main1} on the number of points (namely, $r \ge e^2+k+1=10+k)$
does not hold. 
\end{example}

\begin{example} If $e > 3$, our hypotheses in Theorem \ref{main1} are
  not likely to be
  optimal. We will illustrate this with just one
  example. 

Let $e=5, k = 5, r=31$ and consider $L_d =
dH-5(E_1+\ldots+E_{31})$. By Theorem \ref{main1}, $L_{35}$ is 5-very
ample. On the other hand, \cite[Corollary 1.4]{BS1} shows that $L_{32}$ is {\it
  not} $5$-very ample. We do not know if $L_{33}$ or $L_{34}$ is 5-very
ample. Note that the criterion of Beltrametti-Francia-Sommese
(Theorem \ref{bfs-theorem}) can not be applied here, since $N_d=L_d-K_X$ is not
nef for $d < 35$. Indeed, $N_d\cdot C_1 = 5(d+3)-186 < 0$, for $d <
35$. In other words, our method (which is to use Theorem \ref{bfs-theorem}
to show $k$-very ampleness) itself is 
not applicable to $L_{33}$ and $L_{34}$.
\end{example}

{\bf Acknowledgements:} We sincerely thank Tomasz Szemberg 
for carefully reading this paper, pointing out some 
mistakes and making useful suggestions. 
We also thank the referee for pointing out some mistakes and making
numerous suggestions which improved the exposition of the paper.

\bibliographystyle{plain}

\end{document}